\documentclass[11pt]{article}

\usepackage[utf8]{inputenc}  
\usepackage[T1]{fontenc}        
\usepackage{geometry}        
\usepackage[english]{babel} 
\usepackage{cite}
\usepackage{amssymb,  xfrac, stmaryrd, blkarray, multirow, enumerate,
  verbatim,xifthen}   
\usepackage{float}
\usepackage{amsmath} 

\usepackage{fancyhdr}
\usepackage{amscd}
\usepackage{graphicx, color, pinlabel, import}    
\usepackage{amsthm}                 
\usepackage{latexsym}    
\usepackage[all]{xy}

\usepackage{hyperref}

\newtheorem{thm}{Theorem}[section]

\newtheorem*{thm*}{Theorem}

\newtheorem{cor}[thm]{Corollary}
\newtheorem*{cor*}{Corollary}
\newtheorem{prop}[thm]{Proposition}
\newtheorem*{prop*}{Proposition}
\newtheorem{lem}[thm]{Lemma}
\newtheorem*{thmA}{Theorem A}
\newtheorem*{thmB}{Theorem B}
\newtheorem*{thmC}{Theorem C}
\newtheorem*{thmE}{Theorem E}
\newtheorem*{corD}{Corollary D}
\newtheorem*{thmF}{Theorem F}

\theoremstyle{definition}
\newtheorem{definition}[thm]{Definition}

\newtheorem{rmk}[thm]{Remark}

\newtheorem*{rmk*}{Remark}

\newtheorem*{convention*}{Convention}

\def\lquotient#1#2{%
\makeatletter
\lower.6ex\hbox{$#1$}\backslash\raise.3ex\hbox{$#2$}%
\makeatother
}

\def\rquotient#1#2{%
\makeatletter
\raise.6ex\hbox{$#1$}/\lower.2ex\hbox{$#2$}%
\makeatother
}

\newcommand{\bbR}{{\mathbb R}}

\newcommand{\bbZ}{{\mathbb Z}}

\newcommand{\ra}{\rightarrow}

\newcommand{\hra}{\hookrightarrow}

\makeatletter
\newcommand{\subjclass}[2][2010]{%
  \let\@oldtitle\@title%
  \gdef\@title{\@oldtitle\footnotetext{#1 \emph{Mathematics subject classification.} #2}}%
}
\newcommand{\keywords}[1]{%
  \let\@@oldtitle\@title%
  \gdef\@title{\@@oldtitle\footnotetext{\emph{Key words and phrases.} #1.}}%
}
\makeatother

\newcommand{\Address}{{
  \bigskip
  \small

   \textsc{Department of Mathematics, University of Vienna,
Oskar-Morgenstern-Platz 1, 1090 Vienna, Austria.}\par\nopagebreak
  \textit{E-mail address}: \texttt{alexandre.martin@univie.ac.at}

}}

\makeatletter
\newdimen\bibindent
\makeatother

\title{\textbf{On the cubical geometry of Higman's group}}  
\author{Alexandre Martin}

\date{}


\begin{document}
\maketitle

\begin{abstract} 
We investigate the cocompact action of Higman's group on a CAT(0) square complex associated to its standard presentation. We show that this action is in a sense intrinsic, which allows for the use of geometric techniques to study the endomorphisms of the group, and show striking similarities  with mapping class groups of hyperbolic surfaces, outer automorphism groups of free groups and linear groups over the integers. We compute explicitly the automorphism group and outer automorphism group of Higman's group, and show that the group is both hopfian and co-hopfian. We actually prove a stronger rigidity result about the endomorphisms of Higman's group: Every non-trivial morphism from the group to itself is an automorphism. We also study the geometry of the action and prove a surprising result: Although the CAT(0) square complex acted upon contains uncountably many flats, the Higman group does not contain subgroups isomorphic to $\bbZ^2$. Finally, we show that this action possesses features reminiscent of negative curvature, which we use to prove a refined version of the Tits alternative for Higman's group.
\end{abstract}

\medskip

Higman's group was constructed in \cite{HigmanGroup} as the first example of a finitely presented infinite group without non-trivial finite quotients. It is  given by the following presentation: 

$$H:= \langle a_i, i \in \bbZ / 4 \bbZ ~|~ a_ia_{i+1}a_i^{-1}= a_{i+1}^{2}, i \in \bbZ / 4 \bbZ  \rangle.$$

If some partial results about Higman's group are known, the structure of the group and its geometry have remained mostly elusive. The group is known to split as an amalgamated product over several non-abelian free subgroups, and such splittings can be used to derive various properties of the group. For instance, the group has finite asymptotic dimension by work of Bell--Dranishnikov \cite{BellDranishnikovHurewicz}, a property that implies the Novikov conjecture and the coarse Baum--Connes conjecture for the group \cite{YuNovikov, YuCoarseBaumConnes}. Such splittings have also been used to show that the Higman group admits in a sense many quotients. Schupp developed small cancellation tools to show that the group is SQ-universal \cite{SchuppHigmanSQ}, that is, every countable group embeds in a quotient of $H$. Minasyan--Osin   used these splittings to show that the group is acylindrically hyperbolic \cite{MinasyanOsinTrees}, a property which has gained a lot of attention in recent years due to the wide 
variety of groups to 
which it applies (see \cite{OsinAcylindricallyHyperbolic} and examples therein). 

Since its construction, the Higman group has mostly been used as a source of counterexamples in group theory. For instance, Platonov--Tavgen used it to construct a counterexample related to a question of Grothendieck on profinite completions \cite{PlatonovTavgenGrothendieck}. The Higman group has also gained attention in recent years following work of Gromov on sofic groups \cite{GromovSoficGroups}, in relation with the Surjunctivity Conjecture \cite{GottschalkSurjunctivity}. The question of whether every group is sofic is still open and is a topic of active research. As residually finite groups are known to be sofic, a potential counterexample has to be searched for among the rare instances of non-residually finite groups at hand, and Higman's group appears in the list of possible candidates, perhaps even more so due to its status of ``pathological'' group.

The goal of this article is to show that, far from being part of a \textit{mus\'ee des horreurs} of group theory, Higman's group possesses a rich structure and a beautiful geometry. The structure of the group turns out to show intriguing similarities with mapping class groups of hyperbolic surfaces, outer automorphisms of free groups, and linear groups over the integers, three classes of groups whose intertwined geometry has been a driving force in geometric group theory in recent years.\\

The main object of study of this article is the  action of Higman's group on a CAT(0) square complex naturally associated to its standard presentation. The action on that complex is cocompact but not proper: Stabilisers of vertices are isomorphic to the solvable Baumslag--Solitar group $BS(1,2)$ and stabilisers of edges are infinite cyclic. In particular, the complex is not locally finite, as links of vertices and edges are infinite. Although such an action was already known to exist (see for instance \cite[Example II.12.17.(6)]{BridsonHaefliger}), this article is, to the author's knowledge, the first to investigate its geometry in a systematic way. 

In this article,  we show that a completely new phenomenon occurs: While the aforementioned actions of $H$ on trees all depended on a choice of a splitting associated to its standard presentation, the action of $H$ on this CAT(0) square complex $X$ turns out to be \textit{intrinsic}, in that it can be constructed directly from the group without any reference to a presentation, as we now explain. First notice that the $1$-skeleton of $X$ can be made into a directed graph by orienting each edge towards the unique vertex in which the edge stabiliser is undistorted. Now let $\Gamma$ be the simplicial graph whose vertices are maximal subgroups of $H$ isomorphic to the Baumslag--Solitar group $BS(1,2)$, and such that two vertices are joined by an edge precisely when the two associated subgroups have a non-trivial intersection. The action of $H$ on itself by conjugation induces a left action of $H$ on $\Gamma$. Let $\Gamma_{\mathrm{dist}}$ be the subgraph of $\Gamma$ with the same vertex set, and such that an edge of $\Gamma$ belongs to $\Gamma_{\mathrm{dist}}$  if and only if it corresponds to two subgroups $P$, $P'$ such that the intersection $P\cap P'$ is distorted in exactly of the subgroups $P$, $P'$. An edge of $\Gamma_{\mathrm{dist}}$ between subgroups $P, P'$ is oriented towards the unique vertex in which $P'\cap P$ is undistorted.

\begin{thmA}\label{thm:main_intrinsic}
 The directed graph $\Gamma_{\mathrm{dist}}$ with its $H$-action is equivariantly isomorphic to the directed $1$-skeleton of the CAT(0) square complex $X$.
\end{thmA}

In a sense, the graph $\Gamma$ is a completely canonical object associated to $H$,  as it simply is the graph encoding the intersection of maximal subgroups of $H$ of the form $BS(1,2)$, with the natural action of $H$ on it by conjugation. From a geometric viewpoint though, the graph $\Gamma_{\mathrm{dist}}$, although equivariantly quasi-isometric to $\Gamma$, possesses clear advantages over $\Gamma$. First of all, due to the presence of many distorted subgroups in Baumslag--Solitar groups, the flag simplicial complex generated by $\Gamma$ has simplices in every dimension and the action is no longer cocompact. Moreover, the graph $\Gamma_{\mathrm{dist}}$ is the $1$-skeleton of a CAT(0) square complex, a particularly rich and well controlled geometry, while the graph $\Gamma$ does not seem to have a good combinatorial geometry naturally associated to it.

There is a priori no reason to believe that such a graph should be \textit{robust}, in the sense that an automorphism of $H$ would induce an automorphism of $X$. 
A parallel can be made here with the case of a Coxeter group $W$ acting on its Coxeter complex $X_W$. Such a complex also appears as the universal cover of a simplex of groups, and stabilisers of vertices  correspond to conjugates of the maximal special subgroups of $W$. In general though, an isomorphism of $W$ does not induce an automorphism of $X_W$, and the associated rigidity problem for Coxeter groups has been extensively studied (see for instance \cite{CharneyDavisCoxeter}).  

It should perhaps come as a surprise then that the graph $\Gamma_{\mathrm{dist}}$, and thus the square complex $X$, do present such a robust behaviour, as an automorphism of $H$ does indeed induce an automorphism of $X$. In a sense, this robustness makes $\Gamma_{\mathrm{dist}}$ closer to be a \textit{curve complex} for the Higman group, although these complexes turn out to have different geometries, as will be explained below. Such robustness allows for the use of geometric techniques to study the automorphisms, and more generally the endomorphisms, of the group. The most spectacular consequence of this approach is the following rigidity theorem for the endomorphisms of $H$: 

\begin{thmB}\label{thm:main_endo}
 A \textit{non-trivial} morphism from the Higman group to itself is an automorphism. In particular, Higman's group is both hopfian and co-hopfian.
\end{thmB}

Some other classes of groups are known to be both hopfian and co-hopfian. Recall that the Hopf property follows from residual finiteness for finitely generated groups. Residual finiteness was proved by Grossman for mapping class groups and outer automorphism groups of free groups \cite{GrossmanRF},  and follows from a theorem of Malcev for linear groups \cite{MalcevRF}. The co-Hopf property was proved by Ivanov--McCarthy for mapping class groups \cite{IvanovMcCarthyCoHopf}, by Bridson--Vogtmann for $Out(F_n)$ \cite{BridsonVogtmannAutomorphismsOut}, and by Prasad for $SL_n(\bbZ)$, $n \geq 3$ \cite{PrasadLatticesCoHopf}. In another direction, torsion-free one-ended hyperbolic groups are known to be both hopfian and co-hopfian by work of Sela \cite{SelaCoHopfHyperbolic, SelaHopfHyperbolic}, who also obtained deep results about the structure of their endomorphisms.

The following theorem completes Theorem B by giving a full description of the automorphism group of $H$, and pushes the analogy with the likes of mapping class groups even further. It states that, in addition to the inner automorphisms, automorphisms of the group only come from the ``obvious symmetries''.

\begin{thmC}\label{thm:main_auto}
The automorphism group of Higman's group splits as the semi-direct product 
 $$Inn(H) \rtimes  \bbZ / 4\bbZ,  $$
 where $Inn(H) \cong H$ is the inner automorphism group of $H$ and $\bbZ / 4\bbZ$ is the subgroup generated by the automorphism of $H$ permuting cyclically the standard generators. In particular, $Out(H)$ is isomorphic to $\bbZ / 4\bbZ$.
\end{thmC}

Such splittings or results of a similar nature are known to hold for mapping class groups of hyperbolic surfaces by results of Ivanov \cite{IvanovAutomorphisms} and McCarthy \cite{McCarthyAutomorphisms}, and for linear groups over the integers by work of Hua--Reiner \cite{HuaReinerAutomorphismsGLn}. Bridson--Vogtmann showed that the outer automorphism group of $Out(F_n)$ is trivial \cite{BridsonVogtmannAutomorphismsOut}. 

Note that all the tools developed in this article work equally for other squares of  Baumslag--Solitar groups $BS(1,n)$, $n>1$. It is thus possible to modify the presentation to get rid of the obvious symmetries and construct groups with a very constrained set of endomorphisms. 

\begin{corD}
 Let $G$ be the group with presentation 
 $$\langle a,b,c,d ~|~ aba^{-1}= b^2, bcb^{-1}=c^3, cdc^{-1}=d^5, dad^{-1}=a^7\rangle. $$
 Then every non-trivial morphism from $G$ to itself is an inner automorphism.
\end{corD}

The rich geometry of the Higman group and the resulting rigidity phenomena, in addition to being completely unforeseen, are quite astonishing. While the analogies between mapping class groups, $Out(F_n)$, and lattices in Lie groups are well documented,
 the Higman group seems at first to be of a completely different nature. This contrast shows in particular in the definition of the CAT(0) square complex $X$, thought as the analogue of the curve complex for $H$. All the former groups are (outer) automorphism groups of some appropriate objects and the relevant actions are  built using some rich combinatorial structure associated to these objects (actions on curves of a surface, on free splittings of $F_n$, etc.). By contrast, the Higman group is not known to be realised as the automorphism group of some ``simple'' mathematical object whose combinatorics could be used to recover such an action. The fact that the action  of the group on itself by conjugation alone could contain such a rich 
geometry and could display such similarities with well-known classes of groups is somehow mesmerising, and raises the question of whether Higman's group can be realised as the (outer) automorphism group of some interesting combinatorial object.\\

The action of the Higman group on $X$ can also be used to study the subgroups of $H$.  We show that the Higman group satisfies the following refined version of the Tits alternative: 

\begin{thmE}
 A non-cyclic subgroup of Higman's group is either contained in a vertex stabiliser, and hence embeds in $BS(1,2)$,  or contains a non-abelian free subgroup.
\end{thmE}

In particular, the group belongs to the now extensive list of groups satisfying the Tits alternative, which includes linear groups \cite{TitsAlternative}, mapping class groups \cite{McCarthyTitsAlternative, IvanovAutomorphisms}, and outer automorphism groups of free groups \cite{BestvinaFeighnHandelTitsI,BestvinaFeighnHandelTitsII}.

In proving such a theorem, it is necessary to understand the abelian subgroups of $H$. As it turns out, a very surprising phenomenon occurs: 

\begin{thmF}
The CAT(0) square complex $X$ contains uncountably many flats and does not have the Isolated Flats Property. However, none of the flats is periodic, and $H$ does not contain $\bbZ^2$ as a subgroup.
\end{thmF}

To our knowledge, this is the first example of a group acting cocompactly on a CAT(0) space containing isometrically embedded flats, none of which is periodic. In the case of proper actions, the analogous \textit{flat closing conjecture} \cite[Section $6.B_3$]{GromovAsymptotic} is still open in general.

As was pointed out to us by V. Guirardel, one can recover the fact that a subgroup of $H$ is either metabelian or contains a non-abelian free subgroup directly through the action of $H$ on the Bass--Serre tree associated to one of the standard splittings of $H$ \cite{HigmanGroup}. Indeed, this can be done by using the fact that a group acting non-trivially on a simplicial tree either contains a copy of $F_2$ or contains a finite index subgroup fixing a point at infinity. However, ruling out the possibility of a $\bbZ^2$ subgroup fixing a point at infinity in this case would require finer ad hoc arguments to overcome the poorly behaved geometry of such an action. By contrast, the very well controlled geometry of the action of $H$ on $X$ provides us with a very natural approach: The basic properties of the action would force such a $\bbZ^2$ subgroup to act freely on $X$, and the Flat Torus Theorem thus transforms the problem of the non-existence of a $\bbZ^2$ subgroup into a problem about the non-existence of periodic flats in $X$.

Note also that Theorem E would be rather straightforward, should the complex $X$ be hyperbolic and the action of $H$ be acylindrical. While Theorem F tells us that this is far from being the case, it is worth seeing this remark in light of the situation for generalisations of the Higman group. Indeed, as noticed by Higman himself \cite{HigmanGroup}, the Higman group can be generalised to \textit{polygons} of Baumslag--Solitar groups $BS(1,2)$ on $n\geq 4$ generators. In another article \cite{MartinHigmanAcylindrical}, we will show that, for $n \geq 5$, the universal covers of the associated polygons of groups are hyperbolic and that the associated \textit{generalised Higman groups} act acylindrically on their universal covers. Thus, Higman's group can be thought as some sort of limit case in a large family of groups. While negative curvature disappears when reaching $n=4$, the geometry of the action studied in this article indicates that some features of dynamics in  negative curvature do survive this limiting process (see for instance Proposition \ref{prop:free_group}). \\

The article is organised as follows. Section \ref{sec:preliminaries} contains preliminary material about simple squares of groups and the action of $H$ on the associated CAT(0) square complex, and states the combinatorial Gau\ss-Bonnet theorem. In Section \ref{sec:graph}, we introduce the intersection graph and prove Theorem A. In Section \ref{sec:morphisms}, we study morphisms from $H$ to itself  by looking at the induced maps on $X$, and prove Theorems B, C, and D. In Section \ref{sec:action}, we study the action of subgroups of $H$ on $X$ and prove Theorems E and F.\\

\noindent \textbf{Acknowledgements.} The author thanks his colleagues at the University of Vienna, and in particular F. Berlai and M. Steenbock, for stimulating discussions on Higman's group which motivated the present work, and warmly thanks P.-E. Caprace  for useful comments on a preliminary version of this article. The author also thanks D. Allcock for pointing out an elementary proof of Lemma \ref{prop:morphismBS}, and V. Guirardel for pointing out an alternative proof of the Tits alternative for $H$. Finally, the author thanks the anonymous referees for useful comments and suggestions, and for proposing a very useful alternative characterisation of $\Gamma_{\mathrm{dist}}$.
 
 This work was partially supported by the European Research Council (ERC) grant no. 259527 of G. Arzhantseva and by the Austrian Science Fund (FWF) grant M1810-N25.

\section{Preliminaries}\label{sec:preliminaries}

\subsection{Higman's group as a square of groups}

We briefly recall some notions about simple squares of groups, and we refer the reader to \cite[Chapter II.12]{BridsonHaefliger} for more details. Everything contained in this section is fairly standard.\\ 

Let $C$ be a square with edges $e_i, i \in \bbZ / 4 \bbZ $ ordered cyclically and vertices $v_{i,i+1}:= e_i \cap e_{i+1}, i \in \bbZ / 4 \bbZ$. A \textit{simple square of groups} consists of the following data: 
\begin{itemize}
 \item for every face $\tau$ of $C$, a \textit{local group} $G_{\tau}$,
 \item for every inclusion $\tau_1 \subset \tau_2$, an injective \textit{local map} $\psi_{\tau_1, \tau_2}:G_{\tau_2} \hra G_{\tau_1}$, such that for every inclusion $\tau_1 \subset \tau_2 \subset \tau_3$, we have  
 $$ \psi_{\tau_1, \tau_2}\psi_{\tau_2,\tau_3} = \psi_{\tau_1,\tau_3}.$$
\end{itemize}

When the square of groups is \textit{developable}, that is, when the \textit{fundamental group} $G:= \underset{\leftarrow}{\lim} ~G_\tau$ is such that the natural map $G_{\tau} \ra G$ is injective for every face $\tau$ of $C$, one can define the \textit{universal cover} of the square of groups as the following square complex:

$$Y:= G \times C / \sim,$$
where two pairs $(g,x), (g',x')$ of $G \times C$ are  equivalent when $x=x'$ and $g^{-1}g' \in G_{\tau}$, $\tau$ being the unique face containing $x=x'$ in its interior. The group $G$ acts without inversion on $Y$ by left multiplication on the first factor, and we denote by $G_{\tau}$ the pointwise stabiliser of a face $\tau$ of $Y$. The action admits the square $C_0$, image of $1\times C$ in $Y$, as strict fundamental domain. For every face $\tau$ of $C$, we denote by $\tau_0$ its image in $C_0$. There exists a family of isomorphisms $f_{\tau}: G_{\tau} \ra G_{\tau_0}$, for $\tau$ a face of $C$, such that for every inclusion $\tau \subset \tau'$ in $C$, we get 
$$f_{\tau} \psi_{\tau, \tau'} = \iota_{\tau_0, (\tau')_0}f_{\tau'},$$
where $\iota_{\tau_0, (\tau')_0}: G_{(\tau')_0} \hra G_{\tau_0}$ is the inclusion of stabilisers. This allows us to identify $C$ and $C_0$.\\

A geometric criterion ensuring developability is the following \textit{non-positive curvature} condition on the square of groups. Such a criterion was originally introduced in greater generality by Gersten--Stallings to study non-positively curved triangles of groups \cite{GerstenStallings}. 

For every vertex $v_{i,i+1}$ of $C$, consider the map $f_{v_{i,i+1}}:G_{e_i} * G_{e_{i+1}} \rightarrow G_{v_{i,i+1}}$ induced by the maps $\psi_{v_{i,i+1},e_{i}}$ and $\psi_{v_{i,i+1},e_{i+1}}$, and let $l_{v_{i,i+1}}$ be the minimal free product length of an element in the kernel of that map. If $l_{v_{i,i+1}} \geq 4$ for every vertex of $C$, then the square of groups is developable and the piecewise-Euclidean metric on $Y$ obtained by identifying each square of $Y$ with a Euclidean square of side $1$ is a CAT(0) metric. Moreover, for every vertex $v$ of $Y$ in the orbit of $v_{i,i+1}$, the link of $v$ in $Y$ is isomorphic to the quotient of the Bass--Serre tree of the splitting $G_{e_i} * G_{e_{i+1}}$ under the action of the normal subgroup $\ker{f_{v_{i,i+1}}}$. In particular, the girth of the link of that vertex is $l_{v_{i,i+1}}$.\\
 
We now apply this to the case of Higman's group and other squares of solvable Baumslag--Solitar groups. We associate to positive integers $m_i, i \in \bbZ / 4 \bbZ$, a \textit{Higman-like group} $H_{(m_i)_{i \in \bbZ / 4 \bbZ}}$  given by the following presentation: 
$$H_{(m_i)_{i \in \bbZ / 4 \bbZ}} = \langle a_i, i \in \bbZ / 4 \bbZ ~|~ a_ia_{i+1}a_i^{-1}= a_{i+1}^{m_i}, i \in \bbZ / 4 \bbZ  \rangle.$$
The generators $a_i, i \in \bbZ / 4 \bbZ$, will be referred to as the \textit{standard generators}. Such groups can naturally be seen as a non-positively curved squares of groups as follows. The local group associated to the edge $e_i$ is an infinite cyclic group with a chosen generator $a_i$. The local group associated to each vertex $v_{i, i+1}$ is the Baumslag--Solitar group $BS(1,m_i)$. The local group associated to the square is trivial. Finally, the local maps $\psi_{v_{i,i+1},e_{i}}$, $\psi_{v_{i,i+1},e_{i+1}}$ send the generators $a_i, a_{i+1}$ to generators $b_i, b_{i+1}$ of $BS(1,m_i)$ satisfying the relation $b_ib_{i+1}b_i^{-1}= b_{i+1}^{m_i}$. In this case, each $l_{v_{i,i+1}}$ is $4$ and the simple square of groups is thus non-positively curved.

\begin{convention*}
In everything that follows apart from Corollaries \ref{cor:H_2_2_2_2} and \ref{cor:H_2_3_5_7}, we fix positive integers $m_i, i \in \bbZ / 4 \bbZ$. In order to lighten notations, we will simply denote by $H$ the  group $H_{(m_i)_{i \in \bbZ / 4 \bbZ}} $. Such a group will be referred to as \textit{a Higman-like group}.
\end{convention*}

The group $H$ thus acts cocompactly on the  universal cover $X$ of the associated square of groups. The complex $X$ is a CAT(0) square complex on which $H$ acts with a square as a strict fundamental domain. Moreover, the link of every vertex is a bipartite graph of girth $4$ (see Section \ref{sec:links} for further details).

\subsection{The combinatorial Gau\ss--Bonnet theorem for CAT(0) square complexes}

We now recall a combinatorial version of the Gau\ss--Bonnet Theorem introduced in \cite{McCammondWiseFansLadders}. \\

A \textit{disc diagram} $D$ over $X$ is a contractible planar square complex endowed with a combinatorial map  $D \ra X$ that is an embedding on each square. A disc diagram $D$ over $X$ is called \textit{reduced} if no two distinct squares of $D$ that share an edge are mapped to the same square of $X$. A diagram is called \textit{non-degenerate} if its boundary is homeomorphic to a circle, and \textit{degenerate} otherwise. A vertex of $D$ is an \textit{internal} vertex if its link is connected and a \textit{boundary} vertex otherwise. By the Lyndon--van Kampen theorem, every combinatorial loop of $X$ is the boundary of a reduced disc diagram. 

The \textit{curvature} of a vertex $v$ of $D$ is defined as
$$\kappa_D(v) = 2\pi - \pi \cdot \chi(\mbox{link}(v)) - n_v \frac{\pi}{2},$$
where $n_v$ is the number of squares of $D$ containing $v$.
As the link of a vertex of $X$ is a bipartite graph of girth $4$, the curvature of an internal vertex of $D$ is non-positive. The curvature of a boundary vertex is $\frac{\pi}{2}$ if it is contained in a single square of $D$, and non-positive otherwise. A boundary vertex of $D$ is called a \textit{corner} if it has non-zero curvature.

Our main tool in controlling the geometry of $X$ is a combinatorial version of the Gau\ss--Bonnet Theorem (see for instance \cite[Section V.3]{LyndonSchupp}). In this article, we use a formulation due to McCammond--Wise \cite[Theorem 4.6]{McCammondWiseFansLadders}:

\begin{thm}[Combinatorial Gau\ss--Bonnet Theorem {\cite[Theorem 4.6]{McCammondWiseFansLadders}}] For a reduced disc diagram over $X$, we have:
\begin{equation*}
~~~~~~~~~~~~~~~~~~~~~~~~~~~~~~~~~~~~~~\underset{v \mathrm{ ~vertex ~of~ }D}{\sum} \kappa_D(v) = 2\pi. ~~~~~~~~~~~~~~~~~~~~~~~~~~~~~~~~~~~~~~\qed
\end{equation*}
\label{GaussBonnet}
\end{thm}

\section{The intersection graph associated to the Higman group}\label{sec:graph}

In this section, we prove Theorem A.

\subsection{Elementary properties of the action}

The $1$-skeleton of $X$ naturally carries some extra structure, which we now describe. \\

\textbf{Orientation.} By definition of the square of groups, we have that for every edge $e$ of $X$ with vertices $v$ and $v'$, one of the morphisms $G_e \hra G_v$, $G_e \hra G_{v'}$ is conjugated to $\langle a \rangle \hra \langle a,b | aba^{-1}=b^m \rangle$ while the other one is conjugated to  $\langle b \rangle \hra \langle a,b | aba^{-1}=b^m \rangle$, for some integer $m \geq 2$. In particular, one of the morphisms $G_e \hra G_v$, $G_e \hra G_{v'}$ is undistorted while the other  is distorted. This allows us to define an \textit{orientation} for each edge of $X$, by orienting  an edge $e$ of $X$ towards the unique vertex $v$ of $e$ such that the inclusion $G_e \hra G_v$ is undistorted. Such a choice of orientations yields an orientation of the boundary path of each square of $X$. The action of $H$ preserves this orientation. 

\begin{figure}[H]
	\begin{center}
		\scalebox{0.8}{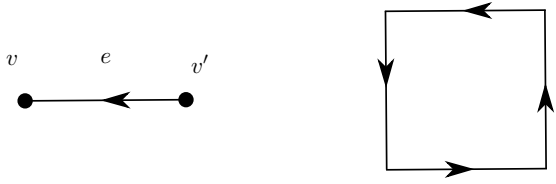}
		\caption{Left: Orientation of an edge of $X$, with the inclusion $G_e \hra G_v$ undistorted.  Right:  Orientation of the edges on the boundary of a square of $X$.}
		\label{figure}
	\end{center}
\end{figure}

\textbf{Type.} Since $H$ acts on $X$ with a square as a strict fundamental domain, there are exactly four orbits of edges.   The \textit{type} of an edge $e$ of $X$ is the unique element $i \in \bbZ/4\bbZ$ such that the stabiliser of $e$ is conjugated to $\langle a_i \rangle$. The action of $H$ on $X$ preserves the type of edges.\\

We now  give a few basic facts about the action of $H$ on $X$. For an element $h$ of $H$ acting elliptically on $X$, the \textit{fixed point set} $\mbox{Fix}(h)$ is defined as the subcomplex of $X$ spanned by faces that are pointwise fixed by $h$. Since $H$ acts freely on the squares of $X$ and the square complex $X$ is CAT(0), the fixed point sets are combinatorially convex subgraphs of $X$. 

\begin{lem}\label{lem:weak_acylindricity}
 Let $v$ be a vertex of $X$ and $e$ an edge containing it, such that the embedding $G_e \hra G_v$ is undistorted. Then for every other edge $e'$ containing $v$, we have $G_e \cap G_{e'} = \{1\}$.
\end{lem}

\begin{proof}
 The lemma amounts to proving the following. Let $\langle a,b | aba^{-1}=b^m \rangle$ be a presentation for the Baumslag--Solitar group $BS(1,m)$, $m\geq 2$. Then for every $g \in BS(1,m) \setminus \langle a \rangle$, we  have $g\langle a \rangle g^{-1} \cap \langle a \rangle = \{1\}$, and for every $g \in BS(1,m)$, we  have $g\langle b \rangle g^{-1} \cap \langle a \rangle = \{1\}$. 
 
 The first equality follows from the malnormality of the cyclic subgroup generated by the stable letter of $BS(1,m)$, which can be seen from the action of the group on its Bass--Serre tree.
 For the second equality, notice that an equality of the form $g b^k g^{-1} = a^l $ in $BS(1,m)$ yields $a^l= 1$ in $\langle a \rangle$ after abelianising. Thus $l=0$ and $g b^k g^{-1} = a^l = 1$ in $BS(1,m)$.
\end{proof}

\begin{cor}[weak acylindricity]\label{cor:acylindricity}
 For every two vertices $v,v'$ at combinatorial distance at least $3$ in $X$, we have $G_v \cap G_{v'}=\{1\}.$
\end{cor}

\begin{proof}
 Recall that fixed point sets are connected subgraphs of $X$. If an element $h$ of $H$ stabilises a sequence $e_1, e_2, e_3$ of three adjacent edges, then for some $i \in \{1,3\}$ we have $G_{e_2} \cap G_{e_i} = \{1\}$ by Lemma \ref{lem:weak_acylindricity} and thus $h$ is trivial.
\end{proof}

\begin{cor}\label{cor:Fix}
 Let $h$ be a non-trivial element of $H$ acting elliptically on $X$. Then the fixed point set $\mbox{Fix}(h)$ is either a single vertex or a reunion of edges contained in the star of a single vertex, pointing towards its degree $1$ vertices, as depicted in Figure \ref{fig:fix}. In the latter case, the angle at every two edges of $\mbox{Fix}(h)$ is a non-zero multiple of $\pi$. \qed
 \end{cor}

\begin{figure}[H]
\begin{center}
\scalebox{1.2}{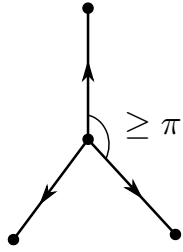}
\caption{A few oriented edges in the fixed point set of an element of $H$ stabilising an edge of $X$.}
\label{fig:fix}
\end{center}
\end{figure}

\subsection{The intersection graph}

We now prove Theorem A. In order to do that, we first need to control how solvable Baumslag--Solitar groups of the form $BS(1,m)$, $m\geq 2$ can be mapped to $H$.

\begin{lem}\label{prop:morphismBS}
 Let $m\geq 2$ be an integer and let $f:BS(1,m) \ra H$ be a non-trivial morphism. Then the image of $f$ is contained in a special subgroup of $H$. 
\end{lem}

\begin{proof}
Let $a,b$ be generators of $BS(1,m)$ satisfying $aba^{-1}=b^m$. 
Since $X$ is a CAT(0) square complex, isometries of $X$ are either elliptic or hyperbolic by a result of Bridson \cite[Theorem A]{BridsonSemiSimple}. As $b$ generates a distorted subgroup of $BS(1,m)$, $f(b)$ is necessarily elliptic. The relation $aba^{-1}=b^m$ implies that $\mbox{Fix}(f(b))\subset \mbox{Fix}(f(b)^m) =f(a) \mbox{ Fix}(f(b))$. If $\mbox{Fix}(f(b))$ is reduced to a single vertex, this vertex is stabilised by $f(a)$. Otherwise, it follows from Corollary \ref{cor:Fix} that $\mbox{ Fix}(f(b))$ and $\mbox{Fix}(f(b)^m)=f(a) \mbox{ Fix}(f(b))$ are trees with the same centre. Thus, $f(a)$ stabilises the centre of $\mbox{Fix}(f(b))$. In both cases, $f(a)$ and $f(b)$ fix a common vertex, and the image of $f$ is thus contained in a special subgroup.
\end{proof}

\begin{cor}\label{cor:maximal_BS}
 The maximal subgroups of $H$ isomorphic to a Baumslag--Solitar group of the form $BS(1,m)$, $m\geq 2$, are exactly the special subgroups of $H$.
\end{cor}

\begin{proof}
Let $f:BS(1,m) \hra H$, $ m \geq 2$, be a monomorphism. By Lemma \ref{prop:morphismBS}, the image of $f$ is contained in a special subgroup of $H$. As $BS(1,m)$ cannot be mapped injectively in an edge stabiliser (which is infinite cyclic), such a special subgroup is necessarily unique.
\end{proof}

\begin{definition}[Intersection graph $\Gamma$]\label{def:graph_Gamma}
Let $\Gamma$ be the simplicial graph whose vertices are maximal subgroups of $H$ isomorphic to a Baumslag--Solitar group of the form $BS(1,m)$, $m\geq 2$, and such that two vertices are joined by an edge precisely when the two associated subgroups have a non-trivial intersection. The action of $H$ on itself by conjugation induces a left action of $H$ on $\Gamma$.
\end{definition}

By Corollary \ref{cor:maximal_BS}, the map sending a vertex $v$ of $X$ to the special subgroup $G_v$ of $H$ defines an equivariant bijection between the vertex sets  of $\Gamma$ and $X$ respectively. Moreover, it also follows from the weak acylindricity of the action (Corollary \ref{cor:acylindricity}) that this bijection extends to an equivariant quasi-isometry between the $1$-skeleton of $X$ and $\Gamma$. 

By Corollary \ref{cor:acylindricity}, the vertices of $X$ corresponding to special subgroups joined by an edge of $\Gamma$ are at distance at most $2$. The following dichotomy is an immediate consequence of Corollary \ref{cor:Fix}:

\begin{lem}\label{lem:dist}
Let $P_1, P_2$ be two special subgroups of $H$ joined by an edge of $\Gamma$ and $v_1, v_2$ the associated vertices of $X$. We have the following:
\begin{itemize}
\item The  vertices $v_1, v_2$ are at combinatorial distance  $1$ if and only if the intersection $P_1 \cap P_2$ is distorted in exactly one of the special subgroups $P_1, P_2$.
\item The vertices $v_1, v_2$ are at combinatorial distance  $2$ if and only if the intersection $P_1 \cap P_2$ is undistorted in both $P_1$ and $P_2$.\qed
\end{itemize}
\end{lem}

This yields a bi-colouring of $\Gamma$ which is clearly preserved by the action of $H$. This motivates the following definition:

\begin{definition}[Intersection graph $\Gamma_{\mathrm{dist}}$]
 Let $\Gamma_{\mathrm{dist}}$ be the subgraph of $\Gamma$ with the same vertex set, and whose edges are the edges of  $\Gamma$ between special subgroups $P_1, P_2$ such that the intersection $P_1 \cap P_2$ is distorted in exactly one of the special subgroups $P_1, P_2$. We orient an edge between $P_1$ and $P_2$ towards $P_1$ if and only if the inclusion $P_1 \cap P_2 \hra P_1$ is undistorted.
\end{definition}

The $H$-action on $\Gamma$ restricts to an action on $\Gamma_{\mathrm{dist}}$.

\begin{prop}\label{prop:canonical}
 The  map sending a vertex $v$ of $X$ to the special subgroup $G_v$ of $H$ extends to an equivariant isomorphism between the oriented $1$-skeleton of $X$ and $\Gamma_{\mathrm{dist}}$.
\end{prop}

\begin{proof}
We already mentioned that this map is a bijection at the level of vertices. Moreover, it induces a combinatorial map $X^{(1)} \ra \Gamma_\mathrm{dist}$ by construction of $\Gamma_\mathrm{dist}$. Consider two special subgroups $P, P'$ at distance $1$ in $\Gamma_\mathrm{dist}$. By Lemma \ref{lem:dist}, the associated vertices of $X$ are at distance $1$, which concludes the proof.
\end{proof}

\section{Geometry of the morphisms}\label{sec:morphisms}

In this section, we prove Theorems B, C, and D.

It follows immediately from Proposition \ref{prop:canonical} that an automorphism $\alpha$ of $H$ induces an automorphism $\alpha_*$ of $X$. We now describe some automorphisms of $H$ and the induced automorphism of $X$. We start by describing the automorphisms of $X$ associated to inner automorphisms of $H$.

\begin{lem}
 A Higman-like group has a trivial centre. 
\end{lem}

\begin{proof}
 An element in the centre defines an isometry of $X$ which is either elliptic or hyperbolic. If such an element defined a hyperbolic isometry, the fixed point set of an arbitrary elliptic isometry would be stable under the action of this hyperbolic isometry. Such a fixed point set would then be unbounded, which contradicts the weak acylindricity (Corollary \ref{cor:acylindricity}). If such an element is elliptic, it then defines an element in the centre of some $BS(1,m)$, $m\geq 2$. As such a group has a trivial centre, the result follows. 
\end{proof}

\begin{cor}
 The inner automorphism group of $H$ is isomorphic to $H$.\qed
\end{cor}

\begin{lem}
If $\alpha$ is the conjugation by an element $g$ of $H$, then the induced automorphism of $X$ is the action of $g$ on $X$ coming from the action of $H$ on $X$. \qed
\end{lem}

We now show that we can associate a combinatorial map $X \ra X$ to every \textit{non-trivial} morphism $H \ra H$.

\begin{prop}\label{prop:induced_map}
A non-trivial morphism $f: H \ra H$ induces, in a functorial way, a combinatorial map $f_*: X \ra X$ that restricts to an orientation-preserving isomorphism on each square of $X$.
\end{prop}

The proof of Proposition \ref{prop:induced_map} will be split into several lemmas. Before starting the proof, we mention a simple but useful remark. 

\begin{lem}\label{lem:generator_not_killed}
Let $f:H \ra H$ be a non-trivial morphism. Then $f$ does not kill any of the standard generators $a_i$.
\end{lem}

\begin{proof}
 If one of the $a_i$ was killed, using the cyclic relations defining $H$ would imply that each $a_i$ is killed by $f$, and thus would be trivial.
\end{proof}

Let $f:H \ra H$ be a non-trivial morphism.

\begin{lem}\label{lem:induced_0}
 The morphism $f$ sends each special subgroup of $H$ to a unique special subgroup of $H$.
\end{lem}

\begin{proof}
By Lemma \ref{prop:morphismBS}, for each special subgroup $P$ of $H$, the image is contained in a special subgroup $P'$ of $H$. If that special subgroup $P'$ was not unique, then  $f(P)$ would be contained in some edge group, which is abelian. But in a Baumslag--Solitar group $BS(1,m) \cong \langle a,b | aba^{-1}=b^m\rangle$, $m \geq 2$, the defining relation implies that the generator $b$ is killed by $f$ as $H$ is torsion-free, which  would then contradict Lemma \ref{lem:generator_not_killed}. 
\end{proof}

In particular, $f$ induces a map $f_*: X^{(0)} \ra X^{(0)}$. 

\begin{lem}\label{lem:induced_1}
The map $f_*: X^{(0)} \ra X^{(0)}$ extends to a combinatorial map $f_*: X^{(1)} \ra X^{(1)}$ between the directed $1$-skeletons. 
\end{lem}

\begin{proof}
Let $v$, $v'$ be two vertices of $X$ contained in an edge $e$. As $G_e$ is not killed by $f$ by Lemma \ref{lem:generator_not_killed}, we get that $G_{f_*(v)} \cap G_{f_*(v')}$ is non-trivial, hence $f_*(v),f_*(v')$ are at distance at most $2$ by weak acylindricity (Corollary \ref{cor:acylindricity}). Moreover, as $G_e= G_v \cap G_{v'}$ is distorted in one of the special subgroups $G_v, G_{v'}$, it follows that $f(G_v \cap G_{v'})\subset G_{f_*(v)} \cap G_{f_*(v')}$ is distorted in at least one of the special subgroups $G_{f_*(v)}$, $G_{f_*(v')}$. Thus, $f_*(v)$ and $f_*(v')$ are vertices of $X$ at distance exactly $1$ by Lemma \ref{lem:dist}, and it follows that $f_*$ induces a combinatorial map from the $1$-skeleton of $X$ to itself. Moreover, if $G_e= G_v \cap G_{v'}$ is distorted in $G_v$, then $f(G_v \cap G_{v'})\subset G_{f_*(v)} \cap G_{f_*(v')}$ is distorted in $G_{f_*(v)}$, thus $f_*$ preserves the orientation.
\end{proof}

\begin{lem}\label{lem:induced_2}
The map $f_*: X^{(1)} \ra X^{(1)}$ extends to a  combinatorial map $f_*: X \ra X$ that restricts to an isomorphism on each square of $X$.
\end{lem}

\begin{proof}
This amounts to showing that $f_*: X^{(1)} \ra X^{(1)}$ is injective on the set of vertices of the fundamental square of $X$. 
 
First assume that two adjacent vertices of that fundamental square are mapped to the same vertex. Then $f$ would induce a morphism from a subgroup of the form $\langle a_{i-1}, a_i, a_{i+1} \rangle$ to a Baumslag--Solitar group of the form $BS(1,m)$, $m \geq 2$, and in particular a morphism between their derived subgroups. The relations at play imply that $a_i, a_{i+1}$ are both in the derived subgroup of $ \langle a_{i-1}, a_i, a_{i+1} \rangle $, so in particular $f$ would induce a morphism from $\langle a_i, a_{i+1} \rangle$ to the derived subgroup of $BS(1,m)$, the latter group being abelian and torsion-free. The relation between $a_i$ and $a_{i+1}$ now implies that $a_{i+1}$ is killed by $f$, contradicting Lemma \ref{lem:generator_not_killed}.

 Let us now assume that two opposite vertices of that fundamental square are mapped under $f_*$ to the same vertex of $X$. Thus $f$ induces a morphism from $\langle a_i, i\in \bbZ / 4\bbZ \rangle \cong H$ to a Baumslag--Solitar group of the form $BS(1,m)$, $m \geq 2$. The relations between the generators imply that each $a_i$ is in the derived subgroup of $H$, hence $f$ induces a morphism between $H$ and the derived subgroup of $BS(1,m)$, which is abelian and torsion-free. As before, this implies that all the generators are killed by $f$, a contradiction.  This concludes the proof.
\end{proof}

\begin{proof}[Proof of Proposition \ref{prop:induced_map}] The result follows from Lemmas \ref{lem:induced_0}, \ref{lem:induced_1} and \ref{lem:induced_2}.
\end{proof}

\begin{lem}\label{lem:identity_square}
Let $f: H \ra H$ be a non-trivial morphism. If $f_*$ fixes pointwise the fundamental square of $X$, then $f$ is the identity.
\end{lem}

\begin{proof} Let $f$ be as in the statement of the lemma. There exist non-zero integers $n_i, i \in \bbZ/4\bbZ$, such that 
 
 $$f(a_0)=a_0^{n_0}, f(a_1)=a_1^{n_1}, f(a_2)=a_2^{n_2} \mbox{ and } f(a_3)=a_3^{n_3}.$$

 Assume by contradiction that $n_i<0$ for some $i$, and write $n_{i+1}=m_{i}^{k_{i}}l_{i}$, with $k_{i} \geq 0$ and $l_{i}$ a non-zero  integer not divisible by $m_i$. There are two cases to consider. 
 
 First assume that $k_{i} < -n_i$. Then $f$ sends the trivial element $a_ia_{i+1}a_i^{-1}a_{i+1}^{-m_i}$ of $H$ to
 $$a_i^{n_i+k_{i}}a_{i+1}^{l_{i}}a_i^{-n_i-k_{i}}a_{i+1}^{-m_in_{i+1}},$$
 which is non-trivial, by applying Britton's Lemma in the group $\langle a_i, a_{i+1}\rangle$ seen as the HNN extension $\langle a_{i+1}\rangle *_{\langle a_i\rangle}$, a contradiction.
 
 Now assume that $k_{i} \geq -n_i$. Then $f$ sends the trivial element $a_ia_{i+1}a_i^{-1}a_{i+1}^{-m_i}$ of $H$ to
 $$a_{i+1}^{m_i^{n_i+k_{i}}l_{i}-m_in_{i+1}} = a_{i+1}^{m_i^{n_i+k_{i}}l_{i}-m_i^{1+k_i}l_i}$$
 which is non-trivial as $n_i<0$, a contradiction. 
 
 Thus $n_i>0$ for every $i$ and $f$ sends the trivial element $a_ia_{i+1}a_i^{-1}a_{i+1}^{-m_i}$ of $H$ to $$a_i^{n_i}a_{i+1}^{n_{i+1}}a_i^{-n_i}a_{i+1}^{-m_in_{i+1}} = a_{i+1}^{m_i^{n_i}n_{i+1}- m_in_{i+1}},$$ and thus $n_{i}=1$.

 We thus have $n_i=1$ for each $i$ and $f$ is the identity map.
\end{proof}

We now give a complete description of the endomorphisms and automorphisms of a Higman-like group. Let $K$ be the finite subgroup of automorphisms of $H$ of the form $a_i \mapsto a_{i+k}$ for some $k$ in $\bbZ/4\bbZ$ such that $m_i = m_{i+k}$ for every $i \in  \bbZ / 4\bbZ$.

\begin{prop}\label{prop:endo_auto}
 A non-trivial morphism $H\ra H$ is an automorphism. Moreover, the automorphism group of $H$ is isomorphic to the semi-direct product $H \rtimes K$.
\end{prop}

\begin{proof}
Let $f: H \ra H$ be such a non-trivial morphism. Then $f_*$ sends the fundamental square of $X$ isomorphically to a square of $X$ in an orientation-preserving way by Proposition \ref{prop:induced_map}.  Up to post-composing $f$ by an inner automorphism, we  can assume that $f_*$ induces an orientation-preserving automorphism of the fundamental square, hence $f_*^4$ is the identity on the fundamental square of $X$. By Lemma \ref{lem:identity_square}, it follows that $f_*^4$ is the identity, and thus $f_*$ is an automorphism. 
Let $(e_i)_{i \in  \bbZ / 4\bbZ}$ be the edges of the fundamental square of $X$, such that the standard generator $a_i$ fixes $e_i$. The restriction of $f_*$ to the fundamental square of $X$ induces an automorphism, still denoted $f_*$, of the indexing set $\bbZ / 4\bbZ$ of the edges of the square. Since $f$ is an automorphism of $H$, it induces a bijection of the vertices of $X$, and thus a bijection of the maximal solvable Baumslag--Solitar subgroups of $H$, and it follows that $\langle a_i, a_{i+1} \rangle \simeq \langle a_{f_*(i)}, a_{f_*(i+1)} \rangle$ for every $i \in \bbZ/4\bbZ$. In particular, as two groups of the form $BS(1,m)$, $BS(1, m')$, with $m,m' >1$, are isomorphic if and only if $m = m'$ (as can be seen from their abelianisation), it follows that $m_i = m_{f_*(i)}$ for every $i \in \bbZ/4\bbZ$. Thus, up to post-composing by an element of $K$, we can assume that $f_*$  is the identity on the fundamental square of $X$, and we conclude by Lemma \ref{lem:identity_square}.
 \end{proof}

\begin{cor}\label{cor:H_2_2_2_2}
 The automorphism group of the Higman group $H_{2,2,2,2}$ splits as the semi-direct product 
 $$Inn(H_{2,2,2,2}) \rtimes  \bbZ / 4\bbZ,  $$
 where $Inn(H_{2,2,2,2}) \cong H_{2,2,2,2}$ is the inner automorphism group of the Higman group and $\bbZ / 4\bbZ$ is the subgroup generated by the automorphism permuting cyclically the standard generators. In particular, the outer automorphism group of the Higman group is isomorphic to $\bbZ / 4\bbZ$.\qed
\end{cor}

\begin{cor}\label{cor:H_2_3_5_7}
A non-trivial endomorphism of the Higman-like group 
$$\langle a,b,c,d ~|~ aba^{-1}= b^2, bcb^{-1}=c^3, cdc^{-1}=d^5, dad^{-1}=a^7\rangle$$
is an inner automorphism.\qed
\end{cor}

\section{The geometry of the action}\label{sec:action}

In this section, we study the action of subgroups of $H$ on $X$. In particular, we prove the following: 

\begin{prop}\label{prop:many_flats}
The complex $X$ contains uncountably many flats. Moreover, the complex $X$  does not have the Isolated Flats Property, and some elements of $H$ define non-contracting hyperbolic isometries of $X$. 
\end{prop}

By contrast, we show in Proposition \ref{prop:no_Z2} that $H$ does not contain a subgroup isomorphic to $\bbZ^2$. In particular, none of the flats of $X$ is periodic.

From the action of $H$ on $X$, we also obtain the following refinement of the Tits alternative for $H$: 

\begin{thm}[refined Tits alternative]\label{thm:Tits}
 A non-cyclic subgroup of $H$ is either contained in a special subgroup of $H$, and hence embeds in a Baumslag--Solitar group of the form $BS(1,m)$, $m \geq2$,  or contains a non-abelian free subgroup.
\end{thm}

\subsection{Cycles in links of vertices}\label{sec:links}

We give here a description of the $4$-cycles in the links of  vertices of $X$.

\begin{definition}
For a pair $(C,C')$ of distinct squares of $X$ sharing an edge, we associate a standard generator $a_{C,C'}$ of $H$ and a non-zero integer $n_{C,C'}$ defined as follows: Since $H$ acts freely transitively on the squares of $X$, choose the unique element $h \in H$,  the unique standard generator $a_{C,C'}$ of $H$ and the unique non-zero integer $n_{C,C'}$ such that $C=hC_0$ and $C'=ha_{C,C'}^{n_{C,C'}}C_0$, where $C_0$ is the fundamental square of $X$.  In particular, we have $a_{C,C'}=a_{C',C}$ and $n_{C,C'}=-n_{C',C}$. 
\end{definition}

\begin{definition}[dual graph and labelling]
Let $Y$ be a CAT(0) square complex. The \textit{dual graph} of $Y$ is the graph $Y^*$ defined as follows: Vertices of $Y^*$ correspond to squares of $Y$. For two squares $C, C'$ of $Y$ sharing an edge, we add an oriented edge from $C$ to $C'$. This dual graph is a graph in the sense of Serre  \cite{SerreTrees}, by declaring that the inverse of the oriented edge from $C$ to $C'$ is the oriented edge from $C'$ to $C$.  

If $Y$ is a full subcomplex of $X$, we add a label to each oriented edge of $Y^*$ by giving an oriented edge from $C$ to $C'$ the label $a_{C,C'}^{n_{C,C'}}$. In particular, the label of the inverse of an edge of $Y^*$ is the inverse of the label of that edge. The action of $H$ on $X$ induces an action of $H$ on $X^*$ by label-preserving isomorphisms.

More generally, let $f: Y \ra X$ be a combinatorial map that restricts to an isomorphism on each square of $Y$, and such that distinct squares of $Y$ sharing an edge are sent to distinct squares of $X$. The dual graph $Y^*$ of $Y$ inherits a labelling from the labelling of $X^*$: For two adjacent squares $C, C'$ of $Y$, we set $n_{C,C'}:= n_{f(C), f(C')}$ and give the oriented edge from $C$ to $C'$ the label $a_{f(C),f(C')}^{n_{f(C),f(C')}}$. Conversely,  the labelled dual graph $Y^*$ together with the image of a single square completely characterises the map $f$.
\end{definition}

A $4$-cycle of the link of $v$ corresponds to a subcomplex $D$ of $X$ isomorphic to a $2 \times 2$ grid, centred at $v$.

\begin{lem}\label{lem:4_cycle}
The dual graph of a $2 \times 2$ grid contained in $X$ is, up to orientation-preserving and label-preserving isomorphisms,  of the following form, for some $i \in \bbZ / 4 \bbZ$ and  positive integers $n_i, n_{i+1}$: 

\begin{figure}[H]
\begin{center}
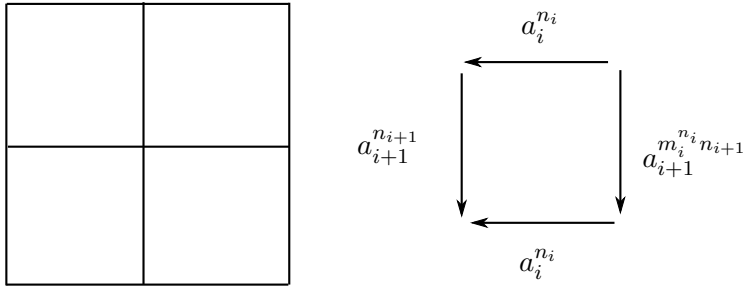
\caption{Left: A $2 \times 2$ grid of $X$. Right: The associated dual graph (for clarity, we only represent one oriented edge for each pair of inverse oriented edges).}
\label{fig:flats}
\end{center}
\end{figure}
\end{lem}

\begin{proof}
Denote by $C_1, \ldots, C_4$ the squares of the grid displayed counter-clockwise, starting from the top-right corner. Since the action of $H$ is free on the squares of $X$, it follows that the product $$a_{C_1, C_2}^{n_{C_1, C_2}}a_{C_2, C_3}^{n_{C_2, C_3}}a_{C_3, C_4}^{n_{C_3, C_4}}a_{C_4, C_1}^{n_{C_4, C_1}}$$ is the identity element. Without loss of generality, we can assume that $a_{C_1,C_2} = a_{C_3,C_4} = a_i$ and $a_{C_2,C_3} = a_{C_4,C_1} = a_{i+1}$ for some $i \in \bbZ / 4 \bbZ$. The result is now a direct consequence of Britton's Lemma for the Baumslag--Solitar group $\langle a_i, a_{i+1} \rangle$ considered as the HNN extension $\langle a_{i+1}\rangle *_{\langle a_i\rangle}$.
\end{proof}

\begin{cor}\label{cor:not_complete}
The link of of a vertex of $X$ is not a complete bipartite graph.
\end{cor}

\begin{proof}
In a complete bipartite graph, every path of length $3$ extends to a cycle of length $4$. It follows from the previous discussion that the path of length $3$ in the link of $v$ defined by the subcomplex $C_0 \cup a C_0 \cup b C_0$ does not extend to a cycle of length $4$, as the dual graph of $C_0 \cup a C_0 \cup b C_0$ cannot be extended to labelled graph of the form given in Lemma \ref{lem:4_cycle}.
\end{proof}

The previous corollary follows from the fact that certain paths of length $3$ in the link of a vertex do not extend to cycles of length $4$. We give an easy criterion that allows for such extensions, which will be used in constructing flats of $X$ by induction.

\begin{lem}\label{lem:extension}
Let $D$ be a $2 \times 2$ grid made of four squares $C_1, \ldots, C_4$ such that $C_i \cap C_{i+1}$ contains an edge for each $i$, and assume that there exists an isometric embedding $f: C_1 \cup C_2 \cup C_3 \hra X$. If $n_{C_1, C_2}$ and $n_{C_3, C_2}$ are both positive, then $f$ extends to an isometric embedding $D \hra X$ such that $n_{C_4, C_1}$ and $n_{C_4, C_3}$ are also positive.
\end{lem}

\begin{proof}
Let $i \in \bbZ / 4 \bbZ$ such that edges containing the image of the centre of the grid are of type $i$ or $i+1$. By Lemma \ref{lem:4_cycle}, the possible extensions are given below by means of their dual graphs, depending on the type of the edge $C_1 \cap C_2$:

\begin{figure}[H]
\begin{center}
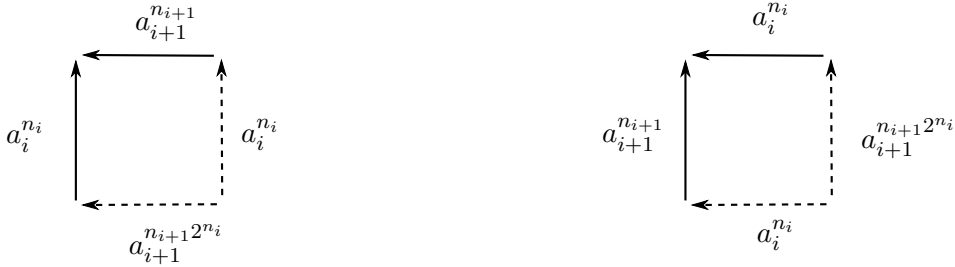
\caption{Dual graphs associated to the two possible extensions. Full arrows correspond to the restriction of $f$ to $ C_1 \cup C_2 \cup C_3$.}
\label{fig:flats}
\end{center}
\end{figure}
\end{proof}

\subsection{Flats in the square complex}\label{sec:flats}

We now construct families of isometrically embedded flats in $X$. We start by giving a procedure to construct local isometries from the Euclidean plane tiled by squares to $X$. This will be done in several steps.

We consider the Euclidean plane $\bbR^2$ tiled by squares centred at points with integral coordinates. A square with centre $(n,m), n, m \in \bbZ$ will be denoted $C_{n,m}$. 
 Let $Q$ be the upper right quadrant, that is, the union of all the squares of the form $C_{n,m}$ with $n, m \geq 0$.  Let $\partial Q \subset Q$ be the subcomplex of $Q$ which is the union of all the squares of the form $C_{0,m}$ with $m \geq 0$ and $C_{n,0}$ with $n \geq 0$.
 
  To each pair of sequence $(k_n)_{n \geq 1}, (k_n')_{n \geq 1}$ of positive integers and each $i \in \bbZ / 4 \bbZ$, we will associate a local isometry $f_{i, (k_n), (k_n')}: Q \ra X$. We first define a local isometry $f_{i,(k_n), (k_n')}: \partial Q \ra X$ as follows.\\

\noindent \textbf{Construction of a local isometry $\partial Q \ra X$.} We start by assigning to each edge of $Q$ a \textit{type} in $\bbZ /4 \bbZ$: The edge $C_{0,0} \cap C_{1,0}$ is given the type $i\in \bbZ /4 \bbZ$,  $C_{0,0} \cap C_{0,1}$ is given the type $i+1\in \bbZ /4 \bbZ$, and so on cyclically. Using the subgroup of isometries of $\bbR^2$ generated by the orthogonal reflections across the four sides of $C_{0,0}$, we can then assign a type to all the edges of $Q$. 

We now construct the map $f_{i,(k_n), (k_n')}: \partial Q \ra X$ by defining it on each square of $\partial Q$. The restriction to $f_{i,(k_n), (k_n')}$ to the square $C_{2n,0 }$ (respectively $C_{2n+1,0 }$, $C_{0,2n }$, $C_{0,2n+1 }$) is the unique type-preserving isometry such that: 
\begin{itemize}
\item $f_{i,(k_n), (k_n')}$ sends the square $C_{0,0}$ to the fundamental square $C_0$,
\item $f_{i,(k_n), (k_n')}$ sends the square  $C_{2n,0 }$ to the square  $\big(a_{i}^{-k_1}a_{i+2}^{-k_2}a_{i}^{-k_3}a_{i+2}^{-k_4} \cdots a_{i}^{-k_{2n-1}}a_{i+2}^{-k_{2n}} \big)C_0 $,
\item $f_{i,(k_n), (k_n')}$ sends the square  $C_{2n+1,0}$ to the square $\big(a_{i}^{-k_1}a_{i+2}^{-k_2}a_{i}^{-k_3}a_{i+2}^{-k_4} \cdots a_{i}^{-k_{2n+1}}\big) C_0 $,
\item $f_{i,(k_n), (k_n')}$ send the square  $C_{0,2n }$ to the square $\big(a_{i+1}^{-k_1'}a_{i+3}^{-k_2'}a_{i+1}^{-k_3'}a_{i+3}^{-k_4'} \cdots a_{i+1}^{-k_{2n-1}'}a_{i+3}^{-k_{2n}'} \big)C_0 $,
\item $f_{i,(k_n), (k_n')}$ sends the square  $C_{0,2n+1}$ to the square  $\big(a_{i+1}^{-k_1'}a_{i+3}^{-k_2'}a_{i+1}^{-k_3'}a_{i+3}^{-k_4'} \cdots a_{i+1}^{-k_{2n+1}'} \big)C_0 $.
\end{itemize}
 One easily checks that our definition of the type of edges of $Q$ allows us to glue all these maps together into a local isometry $f_{i,(k_n), (k_n')}: \partial Q \ra X$.
\\

\noindent \textbf{Extension to $Q$}. We wish to extend $f_{i,(k_n), (k_n')}$ to a local isometry $Q  \ra X$ by extending it square by square by induction. We start from the leftmost square of the horizontal hyperplane consisting of squares of the form $C_{n,1}$, $n \geq 1$ and we move to the adjacent square to the right. Once we have extended $f$ to all the squares on a given horizontal hyperplane, we move to the hyperplane above and start again from the leftmost square on which $f_{i,(k_n), (k_n')}$ is not already defined, moving to the right through adjacent squares. At each stage of the process, when trying to extend $f_{i,(k_n), (k_n')}$ to a square of the form $C_{l,m}$, $f_{i,(k_n), (k_n')}$ has already been defined on $C_{l-1,m-1}, C_{l-1,m},$ and $ C_{l, m-1}$. Moreover, as can be shown by induction by using Lemma \ref{lem:extension}, we have that $n_{C_{l-1,m}, C_{l-1,m-1}}$ and $n_{C_{l,m-1}, C_{l-1,m-1}}$ are both positive. Thus, Lemma \ref{lem:extension} implies that $f_{i,(k_n), (k_n')}$ extends to $C_{l,m}$ and the extended application is again a local isometry.

This procedure yields a local isometry $f_{i,(k_n), (k_n')}: Q \ra X$. \\

\noindent \textbf{Construction of an embedded flat.} Now consider four sequences of positive integers $(k_n^{(i)})_{n \geq 1}, i \in \bbZ / 4 \bbZ$. Each triple $(i, (k_n^{(i)})_{n \geq 1}, (k_n^{(i+1)})_{n \geq 1})$ yields a local isometry $f_{i, (k_n^{(i)}), (k_n^{(i+1)})}: Q^{(i)} \ra X $, where $Q^{(i)}$ is a copy of the quadrant $Q$, with the type of its edges being defined as before. We define a new square complex $Y$ from the disjoint union of the quadrants $Q^{(i)}, i \in \bbZ / 4 \bbZ$, by identifying the square $C_{n,0}^{(i)}$ of $Q^{(i)}$ with the square $C_{0,n}^{(i+1)}$ of $Q^{(i+1)}$ in a type-preserving way. The complex $Y$ is isomorphic to the square tiling of the Euclidean plane. Moreover, the local isometries $f_{i, (k_n^{(i)}), (k_n^{(i+1)})}: Q^{(i)} \ra X$ yield a local isometry $f: Y \ra X$. To conlude that $f$ is an isometric embedding, we need the following lemma:

\begin{lem}[{\cite[Lemma 2.11]{HaglundWiseSpecial}}]\label{lem:embedded_diagram}
 Let $\varphi:X_1 \ra X_2$ be a combinatorial map between CAT(0) square complexes, such that $\varphi$ sends the link of a vertex $v$ of $X_1$ injectively to a full subgraph of the link of $\varphi(v)$. Then  $\varphi$ is an isometric embedding. In particular, its image is a combinatorially convex subcomplex of $X$. \qed
\end{lem}

Since the link of vertices of $X$ do not contain cycles of length $3$ by the CAT(0) condition, it follows that $f:Y \ra X$ satisfies the hypotheses of Lemma \ref{lem:embedded_diagram}. Thus, the image of $f$ is an isometrically embedded flat of $X$.

\begin{proof}[Proof of Proposition \ref{prop:many_flats}]  For each choice of sequences $(k_n^{(i)})_{n \geq 1}, i \in \bbZ / 4 \bbZ$, the above procedure yields an isometrically embedded flat of $X$. If $(k_n^{(0)})_n, (k_n^{(1)})_n, (k_n^{(2)})_n$ are fixed sequences and $(k_n^{(3)})_n$ is allowed to vary among all possible sequences of positive integers, we obtain an uncountable family of isometrically embedded flats which all share a common half-flat. If now we set $(k_n^{(i)})$ to be the sequence constant at $1$, then the flats we obtain contains all squares of the form $(a_i^{-1}a_{i+2}^{-1})^nC_0, n \geq 0$. In particular, we see that the element $a_i^{-1}a_{i+2}^{-1}$ does not define a contracting hyperbolic isometry of $X$.
\end{proof}

\begin{rmk}
 Note however that some elements of $H$ do act on $X$ as contracting isometries, as was pointed out to us by Pierre-Emmanuel Caprace. Indeed, first notice that the action of $H$ on $X$ is essential. Moreover, $X$ does not split as the product of two trees by Corollary \ref{cor:not_complete}. Notice also that $H$ does not globally fix a point in the visual boundary of $X$, as the action is cocompact. Therefore, it follows from work of Caprace--Sageev that $H$ contains \textit{many} elements acting on $X$ as contracting isometries (see \cite[Theorem 6.3]{CapraceSageevRankRigidity} and details therein).
\end{rmk}

\subsection{Generating free subgroups from elliptic isometries}
We now study the action of subgroups of $H$ on $X$. We start by subgroups generated by elements of $H$ acting elliptically on $X$. The aim of this sub-section is to prove the following: 

\begin{prop}\label{prop:free_group}
 Let $a$ and $b$ be two elements of $H$ acting elliptically on $X$. If the  fixed point sets $\mbox{Fix}_{}(a)$ and  $\mbox{Fix}_{}(b)$ are disjoint, then the subgroup generated by $a$ and $b$ contains a non-abelian free subgroup. 
\end{prop}

Although the complex $X$ is far from being hyperbolic by Proposition \ref{prop:many_flats}, Proposition \ref{prop:free_group} shows that the action of $H$ on $X$ shares features of group actions on hyperbolic spaces. We  first reduce the proof to a more tractable case.

\begin{definition}
 The \textit{stable} fixed point set of an element $h \in H$ is defined as 
 $$\mbox{Fix}_{\infty}(h) := \bigcup_{n \geq 1} \mbox{Fix}(h^n).$$
\end{definition}

Let $a$ and $b$ be two elements of $H$ acting elliptically on $X$, such that the  fixed point sets $\mbox{Fix}_{}(a)$ and  $\mbox{Fix}_{}(b)$ are disjoint. We first show that it is enough to prove Proposition \ref{prop:free_group} when the stable fixed point sets $\mbox{Fix}_{\infty}(a)$ and  $\mbox{Fix}_{\infty}(b)$ are disjoint. Let us assume that the fixed point sets $\mbox{Fix}_{}(a)$ and  $\mbox{Fix}_{}(b)$ are disjoint but the stable fixed point sets do intersect. Without loss of generality, let us assume that that the inclusion $\mbox{Fix}_{}(a) \subset \mbox{Fix}_{\infty}(a)$ is strict. By the description of fixed point sets given in Lemma \ref{cor:Fix}, there are only two possible configurations: either $\mbox{Fix}_{\infty}(b)=\mbox{Fix}_{}(b)$ is a single vertex contained in  $\mbox{Fix}_{\infty}(a) \setminus \mbox{Fix}_{}(a) $, or the inclusion $\mbox{Fix}_{}(b) \subset \mbox{Fix}_{\infty}(b)$ is strict, and the two stable fixed point sets meet along a 
single vertex. 

In the first case, $b$ and $aba^{-1}$ define two elliptic isometries of $\langle a,b \rangle$ with disjoint stable fixed point sets. In the second case, we can for instance assume that $\mbox{Fix}_{}(a)$ does not contain the vertex $\mbox{Fix}_{\infty}(a) \cap \mbox{Fix}_{\infty}(b)$. Now observe that $ \mbox{Fix}_{\infty}(b) \cup \mbox{Fix}_{\infty}(a) \cup a\mbox{Fix}_{\infty}(b)$ is a convex subcomplex of the CAT(0) complex $X$, as the angle between two distinct adjacent edges is at least $\pi$, which implies that $b$ and $aba^{-1}$ have disjoint stable fixed point sets.\\ 

We thus assume that the stable fixed point sets $\mbox{Fix}_{\infty}(a)$ and  $\mbox{Fix}_{\infty}(b)$ are disjoint, and we choose vertices $v_a \in \mbox{Fix}_{\infty}(a)$ and $v_b \in \mbox{Fix}_{\infty}(b)$ realising the distance between $\mbox{Fix}_{\infty}(a)$ and $\mbox{Fix}_{\infty}(b)$. Let $k,l \geq 1$ be such that  $v_a \in \mbox{Fix}(a^k)$ and $v_b \in \mbox{Fix}(b^l)$.  We define  the following two elements of $\langle a,b \rangle$: 
$$g := b^la^k \mbox{ and } h:= a^{k}ga^{-k} = a^{k}b^l,$$
and we want to show that $g$ and $h$ generate a free subgroup of $H$.

Choose a geodesic segment $P$ between $v_a$ and $v_b$, and let $e_a$ and $e_b$ be the edges of $P$ containing $v_a$ and $v_b$ respectively. By construction, no power of $a$ stabilises $e_a$ and no power of $b$ stabilises $e_b$.  For a non-zero-integer $m$, let $P_m:=b^{lm}P$ and $L_m := P \cup P_m$. Note that $e_b$ and $b^{lm}e_b$  make an angle which is a non-zero multiple of $\pi$ by construction. We also set $P:= P_1$ and $L:=L_1$. Finally, let $K:= L \cup a^kL$. First remark that for every $s \in S:= \{g,h,g^{-1}, h^{-1}\}$, the sets $K$ and $sK$ intersect. To show that $g$ and $h$ generate a free group, we will show that the orbit of $K$ under $\langle g,h \rangle$ is a tree.

First remark that, since no power of $a$ stabilises $e_a$ and no power of $b$ stabilises $e_b$, the angles at the branching points of $K$ or at the branching points between $K$ and $sK$, $s \in S$ (black dots in Figure \ref{fig:translates}), are non-zero multiples of $\pi$.

\begin{figure}[H]
\begin{center}
\scalebox{0.69}{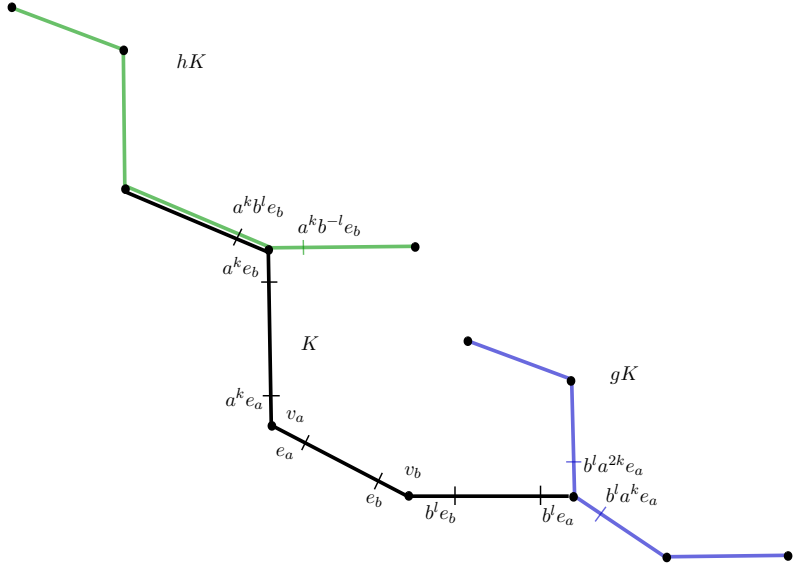}
\caption{The subcomplex $K$ and its translates $gK$ and $hK$.}
\label{fig:translates}
\end{center}
\end{figure}

We also define, for $s \in S$, the following translates of $L$:
$$L_g:=L, L_{g^{-1}}:=g^{-1}L, L_{h}:=a^k L, L_{h^{-1}}:= b^{-l}L,$$
which are used to connect $v_a$ and $sv_a$.
For a reduced word $w=s_1\cdots s_n$, $s_i \in S$, we consider the following subgraph of $X$: 
$$ K_w= K \cup s_1K \cup s_1s_2K \cup \ldots \cup s_1\cdots s_nK.$$

We want to show that for every element $w$ of $F$, the subgraph $K_w$ is a tree. By contradiction, choose a reduced word $w=s_1\cdots s_n$ in the generator system $S$ such that $K_w$ contains an embedded loop and such that each strict subword $w_i:=s_1\cdots s_i, i <n$  defines a tree.  We consider the following  subgraph

$$  L_{s_1} \cup s_1L_{s_2} \cup s_1s_2L_{s_3} \cup \cdots \cup s_1\cdots s_{n-1} L_{s_n}$$

joining $v_a \in L$ and $s_1\cdots s_{n-1}v_a$. Even though we do not know yet that this subgraph is a tree, we think of it as an immersed oriented path from $v_a$ to $w\cdot v_a$. If for some integer $1\leq i \leq n-1$, we have $s_{i} = g^{-1}$ and $s_{i+1}= h^{-1}$, we replace $(s_1 \cdots s_{i-1}) L_{s_i} \cup (s_1\cdots s_i) L_{s_{i+1}}$ by $s_1 \cdots s_{i-1} L_{-2}$ to remove backtracking (see Figure \ref{fig:backtracking}). Similarly, if for some integer $1\leq i \leq n-1$, we have $s_{i} = h$ and $s_{i+1}=g$, we replace $(s_1 \cdots s_{i-1}) L_{s_i} \cup (s_1\cdots s_i) L_{s_{i+1}}$ by $s_1 \cdots s_{i-1}a^k L_{-2}$.

\begin{figure}[H]
\begin{center}
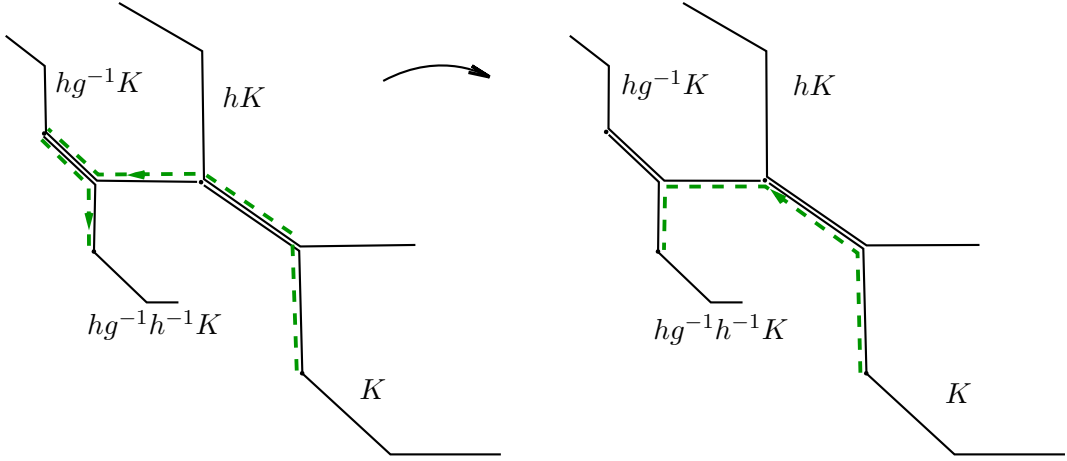
\caption{The procedure for removing backtracking for the word $w=hg^{-1}h^{-1}$.}
\label{fig:backtracking}
\end{center}
\end{figure}

First assume that $n \geq 1$. After performing this operation finitely many times, we obtain a path $L_g' = g_1L_{m_1} \cup \ldots \cup g_{n'}L_{m_{n'}}$, for some $n' \leq n$ and $m_1, \ldots, m_{n'} \in \bbZ \setminus \{0\}$, $g_1, \ldots, g_{n'} \in \langle g, h \rangle$, and such that two consecutive oriented paths $g_{i}L_{m_{i}}$ and $g_{i+1}L_{m_{i+1}}$ meet at the vertex $g_{i+1}v_a$, and the angle between the last edge of $g_{i}L_{m_{i}}$ and the first edge of $ g_{i+1}L_{m_{i+1}}$ is a non-zero multiple of $\pi$ (see Figure \ref{fig:backtracking}). Since $g_1L_1$ and $g_{n'}L_{m_{n'}}$ meet, we can find a connected subpath $Q_0$ of $K$ containing $g_1v_a$ and a connected subpath $Q_{n'}$ of $wK$ containing $g_nv_a$, such that the reunion $\gamma:= Q_0 \cup g_1 L_1 \cup \cdots \cup g_{n-1}L_{n-1}\cup Q_n$ defines an embedded loop such that the angle between two consecutive $Q_i$ or $g_iL_i$ is a non-zero multiple of $\pi$. 

If $n=0$, we can choose an embedded loop in $K$ such for all but maybe one vertex of the loop, the angle between the two edges containing it is a non-zero multiple of $\pi$.

We now want to derive a contradiction by considering a reduced disc diagram with that loop as boundary. In order to do that, we will need the following lemmas:

\begin{lem}\label{lem:geodesic_curvature}
 Let $P$ be a finite combinatorial geodesic of $X$ and $D$ a non-degenerate reduced disc diagram whose boundary path contains $P$. Then 
 $$\sum_{v \in \mathring{P}}\kappa_D(v) \leq \frac{\pi}{2}.$$
 Moreover, equality can happen only if the corners of $D$ contained in $\mathring{P}$ which are the closest to the extremities of $P$ both have positive curvature.
\end{lem}

\begin{proof}
 It is sufficient to show that there do not exist two consecutive corners in $\mathring{P}$ with positive curvature such that vertices in between have zero curvature. Suppose by contradiction that this was the case, and let $v_0, \ldots, v_{n}$ be an injective sequence of adjacent vertices of $P$ such that $v_1, v_{n-1}$ are corners of positive curvature, and $\kappa_D(v_i) = 0$ for $1 < i < n-1$. Then $v_1$ is contained in a unique square of $D$, which also contains $v_0$, $v_{n-1}$ is contained in a unique square of $D$, which also contains $v_n$, and each $v_i$, $1<i<n-1$, is contained in exactly two squares. Thus, $D$ (and thus $X$) contains the following subcomplex (shaded in Figure \ref{fig:geodesic_curvature}):

\begin{figure}[H]
\begin{center}
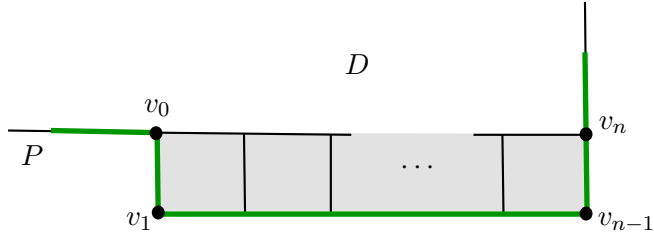
\caption{A configuration contradicting geodesicity.}
\label{fig:geodesic_curvature}
\end{center}
\end{figure}
 
 and we thus get $d(v_0,v_n) < d(v_0) + d(v_1, v_{n-1}) + d(v_{n-1}, v_n)$, which contradicts the fact that $P$ is a combinatorial geodesic.
\end{proof}

\begin{lem}\label{lem:opposite_curvature}
 Let $m$ be a non-zero integer. Let $Q, Q'$ be sub-geodesics of $P, P_m$ respectively containing $v_b$,  such that the concatenation $R:=Q \cup Q'$ is embedded in $X$. Assume that $R$ is contained in an embedded loop of $X$, and let $D$ be a reduced disc diagram with that loop as boundary. Assume that $Q$ and $Q'$ both contain a corner of $D$, and let $v$, $v'$ the corners of $Q$, $Q'$ respectively which are closest to $v_b$. Then $$\kappa_D(v) + \kappa_D(v_b) + \kappa_D(v') \leq 0.$$
\end{lem}

\begin{proof}
Recall that the angle between $e_b$ and $b^{lm}e_b$ is a non-zero multiple of $\pi$. If it is at least $2\pi$, then the desired inequality follows immediately. Let us then assume that the angle is $\pi$, and denote by $v$, $v'$ the corners of $Q$, $Q'$ respectively which are closest to $v_b$. If both corners $v, v'$ have negative curvature, there is nothing to prove, so let us assume that one of the corners has curvature $\frac{\pi}{2}$ in $D$. Without loss of generality, we can assume that $\kappa_D(v)= \frac{\pi}{2}$. We now show $v' = b^{lm}v$ and that $\kappa_D(v') =- \frac{\pi}{2}$.

For that consider the disc diagram $b^{lm}D$. 
Every vertex $u$ between $v$ and $v_b$ has curvature zero in $D$ by definition of $v$, so the same holds for the curvature of $b^{lm}u$ in $b^{lm}D$. Moreover, every vertex between $v_b$ and $v'$ has curvature zero in $D$. This implies that $b^{lm}v=v'$, for otherwise we could construct a embedded cycle of length $3$ in the link of a vertex. 

Let us now assume by contradiction that $\kappa_D(v') = \frac{\pi}{2}$. Since the link of $v'$ does not contain a cycle of length $2$, the squares of $D$ and $b^{lm}D$ containing $v_b$ are the same. This implies that for every edge $e$ of $Q'$ between $v_b$ and $v'$, the squares of $D$ and $b^{lm}D$ containing $e$ are the same. Indeed, this follows by induction as every vertex of $Q'$ between $v_b$ and $v'$ is contained in exactly two squares, and two squares of the CAT(0) square complex $X$ sharing two sides are the same. Thus, $b^{lm}$ stabilises the unique edge $e$ of $D$ containing $v_b$ which is not in the boundary of $D$. But the other vertex of $e$ is thus in $\mbox{Fix}(b^{lm})\subset\mbox{Fix}_{\infty}(b)$ and is strictly closer to $v_a$ than $v_b$, contradicting the definition of $v_a$ and $v_b$.
\end{proof}

Let us now consider a disc diagram $D$ with $\gamma$ as boundary. 

If $n=0$, then the curvature at each but maybe one boundary vertex is non-positive, so we get 

$$\sum_{v \in D} \kappa_D(v) \leq \frac{\pi}{2},$$

contradicting the Gau\ss--Bonnet theorem. Let us now assume that $n \geq 1$. For each $i$, $L_i$ is of the form $P \cup b^{il}P$ with a non-zero multiple of $\pi$ at the branching point. As the link of a vertex of $X$ does not contain an embedded cycle of length $3$, $b^{il}$ restricts to a bijection between the corners of $P$ and those of $ b^{il}P$. It thus follows that $\sum_{v \in \mathring{L}_i} \kappa_D(v)$ is a multiple of $\pi$. Moreover, it follows from Lemma \ref{lem:geodesic_curvature} that $\sum_{v \in \mathring{P}} \kappa_D(v)$, $\sum_{v \in f(b)^{il}\mathring{P}} \kappa_D(v) \leq \frac{\pi}{2}$, and this can only happen if the closest corners to the endpoints of the geodesic have positive curvature. As this cannot happen simultaneously by Lemma \ref{lem:opposite_curvature}, it follows that $\sum_{v \in \mathring{L}_i} \kappa_D(v)< \pi$, and thus $\sum_{v \in \mathring{L}_i} \kappa_D(v) \leq 0$. 

Reasoning analogously on $Q_0$ and $Q_n$, it follows that both $\sum_{v \in \mathring{Q}_0} \kappa_D(v)$ and $\sum_{v \in \mathring{Q}_n} \kappa_D(v)$ are at most $\frac{\pi}{2}$. As all the branching points of $\gamma$ have non-positive curvature and the unique intersection point of $Q_0$ and $Q_n$ brings at most curvature $\frac{\pi}{2}$, it follows that 
$$\sum_{v \in D} \kappa_D(v) \leq \frac{3\pi}{2},$$ contradicting the Gau\ss--Bonnet theorem. Thus, $\langle g,h \rangle$ is a free subgroup of $H$, which concludes the proof of Proposition \ref{prop:free_group}.

\begin{cor}\label{cor:Tits_elliptic}
 Let $G$ be a subgroup of $H$ which contains a non-trivial element acting elliptically on $X$. Then either $G$ fixes a vertex of $X$ or $G$ contains a free subgroup.
\end{cor}

\begin{proof}
  Let $a$ be a non-trivial elliptic element of $G$. If $G$ contains a hyperbolic isometry $b$, then for a sufficiently high power $n \geq 1$, $a$ and $b^nab^{-n}$ are elliptic isometries in $G$ with disjoint fixed point sets. Thus $\langle a,b \rangle$, and hence $G$, contains a non-abelian free subgroup by Proposition \ref{prop:free_group}. If $G$ is purely elliptic, then either $G$ contains a pair of elements with disjoint fixed point sets, in which case $G$ contains a non-abelian free subgroup by Proposition \ref{prop:free_group}, or  the (combinatorially convex) fixed point sets of distinct elements of $G$ pairwise intersect, in which case a theorem \textit{\`a la  Helly} (see for instance \cite{GerasimovSemiSplittings}) implies that they globally intersect, and thus $G$ fixes a vertex of $X$.
\end{proof}

\subsection{The Higman group does not contain $\bbZ^2$}

In order to finish the proof of Theorem \ref{thm:Tits}, we have to understand subgroups of $H$ that are purely loxodromic, that is, such that every non-trivial element of that subgroup acts hyperbolically on $X$. We will need the following result, which is of independent interest: 

\begin{prop}\label{prop:no_Z2}
 A Higman-like group does not contain a subgroup isomorphic to $\bbZ^2$. In particular, no flat of $X$ is periodic.
\end{prop}

\begin{proof}
 By contradiction, let us assume that such a subgroup $G$ exists. First notice that such a subgroup is then necessarily purely loxodromic. Indeed, if $G$ were to contain a non-trivial elliptic element, then Corollary \ref{cor:Tits_elliptic} would imply that $G$ either fixes a vertex of $X$, contradicting the fact that  solvable Baumslag--Solitar groups do not contain free abelian subgroups of rank $2$, or it contains a non-abelian free subgroup, which is absurd.

 Thus, $G$ acts freely by semi-simple isometries on $X$. As $X$ is a CAT(0) square complex, hence has finitely many isometry types of faces, the action is also metrically proper by \cite[Proposition]{BridsonSemiSimple}. The Flat Torus Theorem \cite[Theorem II.7.1]{BridsonHaefliger} thus implies that $X$ contains an isometrically embedded flat $F$ of dimension $2$ on which $G$ acts cocompactly. 
 
 Let $X_F$ be the subcomplex of $X$ generated by $F$, that is, the reunion of all the cubes of $X$ which contain a point of $F$ in their interior. We claim that $F=X_F$. The inclusion $F \subset X_F$ is obvious. If a square $C$ of $X$ contains a point $x\in F$ in its interior, then it also contains a small open ball $B \subset F$ as the interior of $C$ is an open subset of $X$. Take a point $y \in C$, draw the unique geodesic segment between $x$ and $y$, and choose a point $z$ of $B$ on that segment. Let $\bar{x}$ and $\bar{z}$ be the preimages of $x$ and $z$ respectively under the embedding $\bbR^2 \hra X$ with image $F$. Then the geodesic ray starting at $\bar{x}$ and going through $\bar{z}$ maps to a geodesic ray of $X$ containing $x$ and $z$, and thus $y$. Thus $C \subset F$. Let $e$ be an edge of $X$ containing a point $x \in F$ in its interior. Then $e$ contains another point $y \neq x$ of $F$, for otherwise $x$ would be a local cut-point of $F$, which is absurd since $F$ is homeomorphic to $\bbR^2$. The same argument as above now shows that $e \subset F$, which implies that $X_F \subset F$. 
 
 Thus, $X_F$ is a subcomplex of $X$ isometric to $\bbR^2$, in particular it is isomorphic to the standard square tiling of $\bbR^2$. Moreover, $G$ acts cocompactly on that subcomplex.

 Since $G$ acts cocompactly on $X_F$, it follows that the integers
 $ n_{C,C'}$, for adjacent squares $C, C'$ of $X_F$, take only finitely many values. Let $C, C'$ be two adjacent squares of $X_F$ such that $|n_{C,C'}|$ is maximal. Choose the unique vertex $v$ of $e:= C \cap C'$ such that the inclusion $G_e \hra G_v$ is undistorted. Denote the four squares of $X_F$ around $v$ in a cyclic way by $C_1, \ldots, C_4$, in such a way that $e=C_1 \cap C_2$. We have $a_{C_1,C_2}=a_{C_3,C_4}$ and $a_{C_2,C_3}=a_{C_4,C_1}$, and the cyclic product of the $a_{C_i,C_{i+1}}^{n_{C_i,C_{i+1}}}$ is trivial in $G_v$. Moreover, since $a_{C_1,C_2}$ is the stable letter in the standard presentation  associated to $a_{C_1,C_2}$ and $a_{C_2,C_3}$ of the Baumslag--Solitar group $G_v \cong BS(1,m)$ (for some $m \geq 2$), it follows from the normal forms in $BS(1,m)$ (seen as the amalgamated product $\langle a_{C_2,C_3}\rangle *_{\langle a_{C_1,C_2}\rangle}$) that $n_{C_1,C_2}=-n_{C_3,C_4}$, and thus at least one of $n_{C_2,C_3}$ or $n_{C_1,C_4}$ is divisible by $m^{|n_{C,C'}|}$, contradicting the maximality of $|n_{C,C'}|$.
\end{proof}

We are now ready to finish the proof of Theorem \ref{thm:Tits}. 

\begin{proof}[Proof of Theorem \ref{thm:Tits}]
 Let $G$ be a non-cyclic subgroup of $H$. If $G$ contains a non-trivial elliptic isometry, then Corollary \ref{cor:Tits_elliptic} implies that $G$ either acts trivially on $X$ or contains a non-abelian free subgroup.  If $G$ is purely loxodromic, then $G$ acts freely on $X$, and it follows from a theorem of Sageev--Wise \cite[Corollary 1.2]{SageevWiseTits} that $G$ is either virtually abelian or contains a non-abelian free subgroup. As $H$ does not contain a copy of $\bbZ^2$ by Proposition \ref{prop:no_Z2}, it follows that $G$ contains a non-abelian free subgroup, which concludes the proof.
\end{proof}

\bibliographystyle{plain}
\bibliography{Higman_Cubical}

\Address

\end{document}